\documentclass[10pt,a4paper,reqno]{amsart}            

\usepackage[english]{babel}
\usepackage{amssymb,amsmath,amsthm}
\usepackage{mathrsfs}
\usepackage{enumitem}
\usepackage[immediate,debrief]{silence}
\usepackage[pagebackref]{hyperref}

\usepackage{color}

\usepackage{scrtime}

\hypersetup{
 colorlinks,
 linkcolor={blue!90!black},
 citecolor={red!80!black},
 urlcolor={blue!50!black}
}
\usepackage{tikz}
\usepackage{tikz-cd}
\usepackage{graphicx}
\usetikzlibrary{shapes.geometric,arrows,fit,matrix,positioning,trees,snakes}
\tikzset
{
    treenode/.style = {circle, draw=black, align=center, minimum size=1cm},
    subtree/.style  = {isosceles triangle, draw=black, align=center, minimum height=0.5cm, minimum width=1cm, shape border rotate=90, anchor=north}
}

\renewcommand\mathcal{\mathscr}
\theoremstyle{plain}
\newtheorem{theorem}{Theorem}[section]
\newtheorem*{theorem*}{Theorem}
\newtheorem{lemma}[theorem]{Lemma}
\newtheorem*{lemma*}{Lemma}

\newtheorem{proposition}[theorem]{Proposition}

\theoremstyle{remark}
\newtheorem{remark}[theorem]{Remark}
\newtheorem*{remark*}{Remark}

\theoremstyle{definition}

\newtheorem*{definition*}{Definition}

\theoremstyle{example}

\newtheorem*{example*}{Example}

\numberwithin{equation}{section}

\newcount\quantno
\everydisplay{\quantno=0}\everycr{\quantno=0}
\newcommand\quant{\advance\quantno by1
                      \ifnum\quantno=1\qquad\else\quad\fi\forall }
\newcommand\itemno[1]{(\romannumeral #1)}

\newcommand\rest[1]{\kern-.1em
          \lower.5ex\hbox{$\scriptstyle #1$}\kern.05em}

\renewcommand\mod[1]{\vert{#1}\vert}
\newcommand\bigmod[1]{\bigl\vert{#1}\bigr|}
\newcommand\Bigmod[1]{\Bigl\vert{#1}\Bigr|}

\newcommand\norm[2]{{\Vert{#1}\Vert_{#2}}}
\newcommand\normto[3]{{\Vert{#1}\Vert_{#2}^{#3}}}
\newcommand\bignorm[2]{\left.{\bigl\Vert{#1}\bigr\Vert_{#2}}\right.}
\newcommand\bignormto[3]{\left.{\bigl\Vert{#1}\bigr\Vert_{#2}^{#3}}\right.}
\newcommand\Bignorm[2]{\left.{\Bigl\Vert{#1}\Bigr\Vert_{#2}}\right.}

\newcommand\bigopnorm[2]{\big|\!\big|\!\big| {#1} \big|\!\big|\!\big|_{#2}}

\newcommand\wrt{\,\text{\rm d}}

\newcommand\BN{\mathbb{N}}

\newcommand\BZ{\mathbb{Z}}

\newcommand\cB{\mathcal{B}}

\newcommand\cF{\mathcal{F}}   
 
\newcommand\cG{\mathcal{G}}   
 
  \newcommand\fH{\mathfrak{H}}

\newcommand\cK{\mathcal{K}}    
    
\newcommand\cM{\mathcal{M}}     
\newcommand\cN{\mathcal{N}}

  \newcommand\fS{\mathfrak{S}}  
\newcommand\cT{\mathcal{T}}  \newcommand\fT{\mathfrak{T}}  
\newcommand\cU{\mathcal{U}}

\newcommand\al{\alpha}
\newcommand\be{\beta}
\newcommand\ga{\gamma}    
\newcommand\de{\delta}

\newcommand\la{\lambda}   
\newcommand\om{\omega}    \newcommand\Om{\Omega}

\newcommand\funnyk{k\hbox to 0pt{\hss\phantom{g}}}

\newcommand\lu[1]{L^1(#1)}
\newcommand\lp[1]{L^p(#1)}

\newcommand\ly[1]{L^\infty(#1)}

\newcommand\wt{\widetilde}
\newcommand\whH{\widehat{\phantom{G}}\hbox to 0pt{\hss $H$}}

\newcommand\emspace{\hbox to 6pt{\hss}}
\newcommand\ds{\displaystyle}

\newcommand\rmi{\hbox{\rm (i)}}
\newcommand\rmii{\hbox{\rm (ii)}}
\newcommand\rmiii{\hbox{\rm (iii)}}

\newcommand\One{{\mathbf{1}}}

\newcommand\diam{\mathrm{diam}}

\DeclareSymbolFont{EUEX}{U}{euex}{m}{n}

\DeclareSymbolFont{euexlargesymbols}{U}{euex}{m}{n}
\DeclareMathSymbol{\intop}{\mathop}{euexlargesymbols}{"52}
     \def\int{\intop\nolimits}

\DeclareSymbolFont{euexsymbols}     {U}{euex}{m}{n}
\DeclareMathSymbol{\smallint}{\mathop}{euexsymbols}{"52}

\begin{document}

\title[Triangular maximal operators]{Triangular maximal operators \\ on locally finite trees 
}



\thanks{
Santagati is a member of the Gruppo Nazionale per l'Analisi Matematica, la Probabilit\`a e le loro Applicazioni (GNAMPA) of the
Istituto Nazionale di Alta Matematica (INdAM)
}

\author[S. Meda]{Stefano Meda}
\address[Stefano Meda]{Dipartimento di Matematica e Applicazioni \\ Universit\`a di Milano-Bicocca\\
via R.~Cozzi 53\\ I-20125 Milano\\ Italy}
\email{stefano.meda@unimib.it}

\author[F.\ Santagati]{Federico Santagati} \address[Federico Santagati]{Dipartimento di Matematica e Applicazioni \\ Universit\`a di Milano-Bicocca\\
via R.~Cozzi 53\\ I-20125 Milano\\ Italy}
\email{federico.santagati@unimib.it}

\begin{abstract}
	We introduce the centred and the uncentred triangular maximal operators~$\cT$ and $\cU$, respectively, on any locally finite tree in which 
	each vertex has at least three neighbours.  We prove that both $\cT$ and $\cU$ are bounded on~$L^p$ for every $p$ in $(1,\infty]$, that
	$\cT$ is also bounded on $\lu{\fT}$, and that~$\cU$ is not of weak type $(1,1)$ on homogeneous trees.    
	Our proof of the~$L^p$ boundedness of~$\cU$ hinges on the geometric approach of A.~C\'ordoba and R.~Fefferman. 
	We also establish $L^p$ bounds for some related maximal operators.  
	
	Our results are in sharp contrast with the fact that the centred and the uncentred Hardy--Littlewood maximal operators (on balls)
	may be unbounded on~$L^p$ for every $p<\infty$ even on some trees where the number of neighbours is uniformly bounded. 

\end{abstract}


\maketitle

\section{Introduction} \label{s: Introduction}

The centred and the uncentred Hardy--Littlewood maximal operators on a metric measure space $(X,d,\mu)$ are defined by  
\begin{equation} \label{f: HLMO}
\cM f(x)
:= \sup_{r>0} \, \frac{1}{\mu\big(B_r(x)\big)} \, \int_{B_r(x)}\!\! \mod{f} \wrt \mu
\quad\hbox{\textrm{and}}\quad 
\cN f(x)
:= \sup_{B \ni x} \, \frac{1}{\mu(B)} \, \int_{B}  \mod{f} \wrt \mu,
\end{equation}
respectively; here 
$B_r(x)$ denotes the ball with centre $x$ and radius $r$, and  $B$ is any ball in $X$ containing $x$.  

It is well known that if the measure $\mu$ is \textit{doubling}, i.e. if there exists a constant~$D$ such that 
\begin{equation} \label{f: doubling}
	\mu\big(B_{2r}(x)\big) \leq D\, \mu\big(B_{r}(x)\big)
\end{equation}
for every $x$ in $X$ and for all $r>0$, then 
$\cM$ and~$\cN$ are of weak type $(1,1)$ and bounded on $\lp{X}$ for every~$p$ in $(1,\infty]$ (see, for instance, \cite[Chapter~1]{St}).  

If, instead, $\mu$ is nondoubling, viz. the condition~\eqref{f: doubling} fails, then a variety of situations can occur.  
For instance, on symmetric spaces of the noncompact type J.-O.~Str\"omberg \cite{Str} proved that $\cM$ is bounded on $L^p$ for all $p>1$ 
and it is of weak type $(1,1)$, and A.D.~Ionescu \cite{I} showed that $\cN$ is bounded on $L^p$ if and only if $p>2$.  

These results have been complemented by H.-Q.~Li \cite{L1}, who showed that given~$p_0$ in $(1,2)$, there is a nondoubling Riemannian manifold, 
which is a generalisation of the hyperbolic space, where $\cM$ bounded on~$L^p$ if and only if $p$ belongs to the interval $(p_0,\infty]$.  
Furthermore there are Riemannian manifolds of the same type where~$\cM$ is bounded on~$L^p$ if and only if $p=\infty$.  Similar results for $\cN$ 
are contained in \cite{L2}.  See also~\cite{K} and the references therein for simple examples of nondoubling metric measure spaces where~$\cM$ 
and $\cN$ have similar boundedness properties on $L^p$ spaces.   

%
%
%

In this paper we focus on trees: $\fT$ will denote a tree in which every vertex~$x$ has a finite number $\nu(x)\geq 3$ of neighbours.  We emphasize that
the function~$\nu$ may be unbounded on $\fT$, in which case we say that the locally finite tree $\fT$ has \textit{unbounded geometry}.
We endow $\fT$ with the natural graph distance $d$ and the set of its vertices with the counting measure $\mu$.  
For notational convenience, we write~$\mod{E}$ instead of $\mu(E)$ for any subset $E$ of~$\fT$.   

The metric measure space $(\fT,d,\mu)$ has exponential volume growth.  If $\nu$ is bounded, then $\mu$ is locally, but not globally, doubling;
if $\nu$ is unbounded, then~$\fT$ is not even locally doubling.  

In this context, various authors have considered the problem of establishing $L^p$ bounds for $\cM$ and $\cN$.  Notice that the definition
of $\cM$ is usually modified as follows
$$
\cM f(x)
:= \sup_{r\in \BN} \, \frac{1}{\mod{B_r(x)}} \, \int_{B_r(x)} \mod{f} \wrt \mu;
$$
here 
$B_r(x) := \{y\in \fT: d(x,y) \leq r\}$. 
Examples show that the range of $p$'s where either $\cM$ or $\cN$ are bounded on $\lp{\fT}$ 
may depend on the bounds of $\nu$.   

Here is a brief account of some relevant contributions in the literature concerning the $L^p$ boundedness of $\cM$ and $\cN$.  
Recall that a tree where $\nu$ is constant is called \textit{homogeneous}: we denote by $\fT_b$ the tree for which $\nu=b+1$ for some $b\geq 2$.  
A.~Naor and T.~Tao \cite[Theorem~1.5]{NT} and, independently, M.~Cowling, Meda and A.~Setti \cite[Theorem~3.1]{CMS1} proved
that $\cM$ is bounded on $\lp{\fT_b}$, $1<p\leq \infty$, and of weak type $(1,1)$ (see also \cite{RT}).   
A.~Veca~\cite[Theorem~5.1]{V} proved that $\cN$ is bounded on $\lp{\fT_b}$, $2<p\leq \infty$, and of restricted weak type $(2,2)$
(see also the recent work \cite{LS} for results concerning related maximal operators).    

Generalisations of these results to trees $\fT$ where $\nu$ is bounded, but not constant, have been the object of the investigations in \cite{LMSV}.  
In particular, it is shown that if $3\leq a+1\leq \nu \leq b+1$ and $b\leq a^2$, then the precise form of the Kunze--Stein phenomenon on $\fT_b$ 
(see \cite{CMS2}) implies that $\cM$ is bounded on $\lp{\fT}$, $\tau<p\leq \infty$, where $\tau = \log_ab$, 
and it is of restricted weak type $(\tau,\tau)$, and the result is sharp.  If, instead, $b>a^2$, then there are examples of trees in 
this class for which~$\cM$ is unbounded on $L^p$ for every $p<\infty$.
Even more strikingly, whenever $b>a$ there are trees in this class for which~$\cN$ is unbounded on $L^p$ for every $p<\infty$.  

Extensions of some of these results to graphs are contained in \cite{ST}.
We refer the interested reader to the introduction of the paper \cite{LMSV} for additional comments on related works in the literature.  


The abovementioned results concerning the boundedness of $\cM$ and $\cN$ on trees raise the question whether there are natural ``geometric'' maximal 
operators on locally finite trees with possibly unbounded geometry that possess stable $L^p$ boundedness properties, in the sense that the range of $p$'s 
for which they are bounded on~$L^p$ do not depend on the specific assumptions on $\nu$, besides the condition $\nu\geq 3$.    

In this paper we answer in the affirmative to this question and propose to investigate the $L^p$ boundedness of the centred and uncentred
maximal operators on triangles.  

Their definition requires a bit of notation, which we now introduce.  
%
We fix a geodesic ray $\om = \{x_m: m \in \BN\}$ in $\fT$, and consider the associated \textit{height function} $h_\om$, 
which is the discrete analogue of the Busemann function in Riemannian geometry, defined by
$$
h_\om(x) = \lim_{m\to \infty} \big(m - d(x,x_m)\big).  
$$
Notice that $h_\om$ is integer valued.   Its level sets, called \textit{horocycles} associated to $\om$,
are then defined, for $j$ in~$\BZ$, by
$$
\fH_j^\om
:= \big\{x\in \fT : h_\om(x)=j\big\},
$$
and $\ds \fT = \bigcup_{j\in \BZ} \, \fH_j^\om$ (disjoint union).
Notice that if $x\in \fH_j^\om$, then $\nu(x)-1$ neighbours of~$x$, called \textit{successors} of~$x$, 
belong to $\fH_{j-1}^\om$.  
We denote by $s_\om^1(x)$ the set of \textit{successors} of~$x$, 
define $s_\om^0(x) := \{x\}$, and  
$$
s_\om^k(x) 
:= \bigcup_{y\in s_\om^{k-1}(x)} \, s_\om^1(y),
\qquad k\geq 2.
$$   
For every nonnegative integer $R$, we call $\ds T_R^\om(x) := \bigcup_{j=0}^R \, s_\om^j(x)$ the {triangle with vertex $x$ and height~$R$}.  
The \textit{centred and uncentred triangular maximal operators}~$\cT^\om$ and~$\cU^\om$ are then defined by  
$$
\cT^\om f(x) 
:= \sup_{R\geq 0}  \, \frac{1}{\mod{T_R^\om(x)}} \, \int_{T_R^\om(x)} \mod{f} \wrt \mu
\quad\hbox{and}\quad
\cU^\om  f(x)
:= \sup_{T \ni x} \, \frac{1}{\bigmod{T}} \, \int_{T} \bigmod{f} \wrt \mu,
$$
respectively, where $T$ is any triangle in $\fT$.

Our main result states that if $\fT$ is a locally finite tree with $\nu \geq 3$, then $\cT^\om$ and $\cU^\om$ are bounded on $\lp{\fT}$ for every
$p$ in $[1,\infty]$ and for every $p$ in $(1,\infty]$, respectively.  Furthermore, $\cU^\om$ is not of weak type $(1,1)$ even on the homogeneous tree $\fT_b$,
$b\geq 2$.   
 
The operators $\cT^\om$ and $\cU^\om$ depend on $\om$, but their $L^p$ boundedness properties do not, in the sense that if we fix another 
geodesic ray $\om'$, then the corresponding maximal operators $\cT^{\om'}$ and $\cU^{\om'}$ share with $\cT^\om$ and $\cU^\om$, respectively, 
the same~$L^p$ boundedness properties.  Thus, for simplicity, in the sequel we shall omit the superscripts and write~$\cT$ and $\cU$ instead 
of $\cT^\om$ and $\cU^\om$.  

The proof of the $L^p$ boundedness of $\cT$ is not hard, and can be found in Section~\ref{s: Maximal}, where we also study the related centred
and noncentred maximal functions~$\cB$ and~$\cB^u$.  

Our approach to the problem of determining the range of $p$'s where $\cU$ is bounded is much in the spirit of the work of A.~C\'ordoba and 
R.~Fefferman \cite{CF}.  In Section~\ref{s: The uncentred} we show that for every $r$ in $[1,\infty)$ there exists a constant~$A_r$ such that 
for any finite collection $\cG$ of triangles in $\fT$ that are maximal with respect to inclusion the following holds
\begin{equation} \label{f: CF inequality}
	\Bignorm{\sum_{T\in\cG} \One_T}{r}
	\leq A_r \bignorm{\One_G}{r},
\end{equation}
where $G$ denotes the union of all $T$ in $\cG$.   Loosely speaking, this estimate says that the triangles in~$\cG$ have ``finite overlapping
in the~$L^r$ norm''.  
We mention that Ionescu~\cite{I} has used a similar strategy to obtain 
bounds for the uncentred HL maximal operator on symmetric spaces of the noncompact type and rank $\geq 2$.

In Section~\ref{s: Further} we show that \eqref{f: CF inequality} fails for every $r$ in $(1,\infty)$ if we replace the family~$\cG$ above with a family $\cG'$ 
of \textit{modified} maximal triangles $T'$, where $T'$ is the union of a triangle $T$ of height $h$ and 
the $h^{\mathrm {th}}$ ancestor of the vertex of~$T$.  This implies that the uncentred HL maximal operator associated to the family $\cG'$ is unbounded on 
$\lp{\fT}$ for every $p<\infty$.  The reason for which \eqref{f: CF inequality} fails lies in the fact that a point~$x$ can be the~$h^{\mathrm {th}}$ 
ancestor of the vertices of a lot of mutually disjoint triangles of height~$h$, which makes the left hand side, but not the right hand side, 
of \eqref{f: CF inequality} big.  See the observation after Remark~\ref{rem: counterex infinite} for the details.  


\section{Preliminaries}
\label{s: Preliminaries}

Let $\fT$ be a locally finite tree, i.e. a connected graph with no loops, in which every vertex $x$ has a finite number $\nu(x)\geq 3$ of neighbours; we call 
$\nu(x)$ the \textit{valence} of $x$.   

Between any two points $x$ and $y$ in $\fT$, such that $d(x,y)=n$, there is a unique \textit{geodesic path} of the form
$x_0,x_1,\ldots,x_n$, where $x_0=x$, $x_n=y$, and $d(x_i,x_j)=\mod{i-j}$ whenever $0\leq i,j \leq n$.
A \textit{geodesic ray} $\ga$ in~$\fT$ is a one-sided sequence $\{\ga_n:n\in\BN\}$ of points of~$\fT$ such that
$d(\ga_i,\ga_j)=\mod{i-j}$ for all nonnegative integers $i$ and $j$.  We say that {\it $x$ lies on~$\ga$}, and write $x\in \ga$, if
$x=\ga_n$ for some $n$ in~$\BN$. 
Given a point $y$, denote by $y\wedge \om$ the point on $\om$ closest to $y$ ($\om$ is as in the Introduction).  Suppose that $y\wedge \om = x_k$, 
and denote by $\ga_y$ the geodesic ray $[y,x_k]\cup [x_k,x_{k+1},\ldots]$.  Given another point $x$ in $\fT$, we say that 
$x$ lies above~$y$, and write $x\succeq y$, if $x\in \ga_y$.  If $x\succeq y$ and $x\neq y$, then we write $x\succ y$. 

Given a tree $\fT$, we implicitly assume that we have chosen 
a geodesic ray $\om$ in $\fT$. 
Many objects on $\fT$ depend on $\om$.   However, in order to simplify the notation, 
we do not stress this dependence, and write $h$, $T_R(x)$, $s^k(x)$, $\cT$ and $\cU$ in place of 
$h_\om$, $T_R^\om(x)$, $s_\om^k(x)$, $\cT^\om$ and $\cU^\om$.    

We agree that the triangle with vertex $x$ and height $0$ is just the point $x$.  If $T$ is any triangle, then we denote by $v(T)$, 
$h(T)$ and $\be(T)$ its vertex, its height and its base, respectively.  Note that $\be(T) = s^{h(T)}\big(v(T)\big)$.  

Let $x$ be a vertex in $\fT$.  We denote by $p(x)$ the \textit{predecessor} of~$x$, viz. the unique neighbour of $x$ with height $h(x) + 1$.  
Notice that $p(x)$  depends on the choice of~$\om$: in order to simplify the notation, we do not stress this dependence.  
Note that $p\big(p(x)\big)$, also denoted $p^2(x)$, is just a vertex in $\fH_{h(x)+2}$.  The~$k^{\mathrm{th}}$ \textit{ancestor} of~$x$ is the point 
$p^k(x) := p\big(p^{k-1}(x)\big)$.  
For any subset $E$ of $\fT$ and every positive integer $k$, $p^k(E)$ will be short for $\ds \bigcup_{y\in E} p^k(y)$.  

The next lemma contains an elementary inequality relating the area of any triangle in $\fT$ and the length of its base.  Such inequality 
can also be deduced from Cheeger's isoperimetric inequality on trees, for which we refer the reader to \cite[Lemma~13]{RT} and \cite[Theorem~4.2.2]{Wo}.  

\begin{lemma}\label{l: triangles and bases}
	Suppose that $\fT$ is a locally finite tree with $\nu\geq 3$, and let $T$ be a triangle in $\fT$ with height $h$.  The following hold:
	\begin{enumerate}
		\item[\itemno1]
			$2^{k}\, \bigmod{p^{k}\big(\be(T)\big)} \leq \mod{\be(T)}$ for every $k$ in $\{0,\ldots, h\}$;
		\item[\itemno2]
			$\mod{T} \leq 2\, \bigmod{\be(T)}$.
	\end{enumerate}
\end{lemma}

\begin{proof}
	Since every point in $p^{k}\big(\be(T)\big)$ has at least two successors, 
	$$
	\bigmod{p^{k-1}\big(\be(T)\big)} \geq 2\, \bigmod{p^{k}\big(\be(T)\big)}.
	$$
	Then \rmi\ follows by iterating this estimate.  

	Next,
	$$ 
	\mod{T} 
	=    \sum_{k=0}^{h(T)} \, \bigmod{p^k\big(\be(T)\big)} 
	\leq \sum_{k=0}^{h(T)} \, 2^{-k} \, \bigmod{\be(T)}  
	\leq 2\, \, \bigmod{\be(T)},  
	$$
	and \rmii\ follows.  
%
\end{proof}

\section{The centred triangular maximal operator}
\label{s: Maximal}

In this section we study the centred triangular maximal operator $\cT$ defined in the Introduction, and some related maximal operators.  

\begin{theorem}[Centred triangular maximal function] \label{t: cT}
	Suppose that $\fT$ is a tree such that $\nu \geq 3$.  Then $\cT$ is bounded on $\lp{\fT}$ for every $p$ in $[1,\infty]$.
\end{theorem}

\begin{proof}
	Define the function $\tau: \fT\times\fT \to [0,\infty)$ by
	$$
	\tau(x,y)
	:= \frac{1}{\bigmod{T_{d(x,y)}(x)}}\, \One_E (x,y),
	$$
	where $E := \big\{(x,y) \in \fT\times \fT: x\succeq y\big\}$.   Observe that 
	$$
		\cT f(x) 
		\leq \int_\fT \, \sup_{R\in \BN} \, \frac{\One_{T_R(x)}}{\bigmod{T_R(x)}} \, \mod{f} \wrt \mu 
		\leq \int_\fT \, \tau(x, \cdot) \, \mod{f} \wrt \mu.   
	$$
	Therefore 
	$$
		\bignorm{\cT f}{1}
		\leq \int_\fT \wrt \mu(x) \int_\fT \, \tau(x,y) \, \bigmod{f(y)} \wrt \mu(y) 
		\leq A \bignorm{f}{1},  
	$$
	where $\ds A := \sup_{y\in \fT} \, \int_\fT \, \tau(x,y) \wrt \mu(x)$.  Now, given $y$ in $\fT$, the points $x$ for which $\tau(x,y) \neq 0$
	are just the points on the geodesic $[y,\om)$, i.e. the points $y, p(y), p^2(y), \ldots$  Thus
	$$
	A
	= \sup_{y\in \fT} \,\, \sum_{k=0}^\infty \,\, \frac{1}{\bigmod{T_k(p^k(y))}} 
	\leq  \sum_{k=0}^\infty \,\, 2^{-k}
	= 2.  
	$$
	This proves that $\bigopnorm{\cT}{1;1} \leq 2$.  Since $\cT$ is obviously bounded on $\ly{\fT}$, the Marcinkiewicz interpolation theorem
	implies that $\cT$ is bounded on $\lp{\fT}$ for every~$p$ in $[1,\infty]$. 
\end{proof}

An examination of the proof above shows that the assumption $\nu\geq 3$ can be substantially relaxed.  In fact, it suffices to assume that 
$\nu \geq 2$, and that there exists a constant $\la$ in $(0,1)$ such that 
$$
\frac{\bigmod{s^{k-1}(x)}}{\bigmod{s^k(x)}} 
\leq \la
\quant k\geq 1 \quant x \in \fT.
$$

For each function $f$ on $\fT$, define the \textit{centred} and the \textit{uncentred } maximal functions $\cB f$ and $\cB^u f$ by 
$$
\cB f(x)
:= \sup_{r\in \BN} \, \frac{1}{\mod{s^r(x)}} \, \int_{s^r(x)} \!\!\!\mod{f} \wrt \mu
\quad\hbox{and}\quad
\cB^u \! f(x)
:= \sup_{T \ni x} \, \frac{1}{\bigmod{\be(T)}} \, \int_{\be(T)} \bigmod{f} \wrt \mu.  
$$
Clearly $\cB f \leq \cB^u f$.   By Lemma~\ref{l: triangles and bases}~\rmii, applied to $T_r(x)$, $r\geq 0$, 
\begin{equation} \label{f: cS leq cT}
\cB f (x)
	\leq \sup_{T: v(T) = x} \, \frac{2}{\mod{T}} \, \int_{T} \, \mod{f} \wrt \mu
\leq 2 \, \cT \! f (x).   
\end{equation}
The boundedness properties of $\cB$ and $\cB^u$ are grouped together in the next result.  

\begin{theorem} \label{t: main cB}
	The following hold:
	\begin{enumerate}
		\item[\itemno1]
			if $\fT$ is a tree with $\nu \geq 3$, then $\cB$ is bounded on $\lp{\fT}$ for every $p$ in $[1,\infty]$,    
			and $\cB^u$ is bounded on $\lp{\fT}$ for every $p$ in $(1,\infty]$, and satisfies the weak type estimate
			$$
			\bigmod{\big\{x\in \fT: \cB^u\! f(x) > \al \big\}}
			\leq \frac{2}{\al} \bignorm{f}{1} 
			\quant \al >0;
			$$
		\item[\itemno2]
			for every $b\geq 2$, the operator $\cB^u$ is unbounded on $\lu{\fT_b}$. 
	\end{enumerate}

\end{theorem}

\begin{proof}
	Suppose that $\al>0$, and consider, for every $f$ in $\lu{\fT}$, the level set 
	$$
	E_{\cB^u\! f}(\al) := \big\{x\in \fT: \cB^u\! f(x) > \al \big\}.
	$$   
	For notational simplicity, for the duration of this proof we write $E(\al)$ in place of $E_{\cB^u\! f}(\al)$.  

	First we prove \rmi.  The  statement concerning $\cB$ follows from Theorem~\ref{t: cT}  and the pointwise bound \eqref{f: cS leq cT}.

	Next we consider $\cB^u$.  If $z \in E(\al)$, then there exists a triangle $T_z$, containing~$z$, such that 
	\begin{equation} \label{f: Rz}
		\frac{1}{\bigmod{\be(T_z)}} \, \int_{\be(T_z)} \bigmod{f} \wrt \mu >\al.
	\end{equation}
	Now, if $w$ and~$z$ belong to $E(\al)$ and $\be(T_w) \cap \be(T_z) \neq \emptyset$, then either 
	$\be(T_w) \subseteq \be(T_z)$ or $\be(T_w) \supseteq \be(T_z)$. 
	Indeed, $\be(T_w)$ and $\be(T_z)$ are both subsets of the same horocycle, and if $y$ belongs to their intersection, then both
	$v(T_w)$ and $v(T_z)$ (the vertices of~$T_w$ and $T_z$, respectively) must belong to the infinite geodesic $[y,\om)$.   
	Thus, either $v(T_z) \succeq v(T_w)$ or $v(T_w) \succeq v(T_z)$.   

	In the first case $T_z \supseteq T_w$, hence $\be(T_z) \supseteq \be(T_w)$, and in the second $T_z \subseteq T_w$,
	hence $\be(T_z) \subseteq \be(T_w)$.  

	Clearly $E(\al)$ is a union of triangles, because if $E(\al)$ contains $x$, then it contains~$T_x$, where $T_x$ is such that 
	$\ds \frac{1}{\bigmod{\be(T_x)}} \, \int_{\be(T_x)} \bigmod{f} \wrt \mu >\al$.  
	Their size is uniformly bounded, for if $T$ is one such triangle, then Lemma~\ref{l: triangles and bases}~\rmii\ and \eqref{f: Rz} imply that  
	\begin{equation} \label{f: size of triangles cB}
	\mod{T} 
	\leq 2\, \mod{\be(T)}
	< \frac{2}{\al} \, \int_{\be(T)} \bigmod{f} \wrt \mu 
	\leq 2\, \frac{\norm{f}{1}}{\al}.
	\end{equation}
	Thus, $E(\al)$ is the union of a finite number of triangles $T_1,\ldots, T_N$, where, of course, $N$ depends on $\al$.    
	In view of the observation above, we may assume that $\be(T_1),\ldots,\be(T_N)$ are mutually disjoint. 
	Then
	$$
	\bigmod{E(\al)}
	=    \sum_{j=0}^N \,\, \bigmod{E(\al)\cap T_j} 
	\leq \sum_{j=0}^N \,\, \bigmod{T_j}.
	$$
	These estimates, \eqref{f: size of triangles cB} and the disjointness of $\be(T_1),\ldots,\be(T_N)$, imply that 
	$$
		\bigmod{E(\al)}
		< \frac{2}{\al} \, \sum_{j=0}^N \,\,\int_{\be(T_j)} \bigmod{f} \wrt \mu
		\leq  \frac{2}{\al} \bignorm{f}{1} 
		\quant \al >0,
	$$
	as required to prove that $\cB^u$ is of weak type $(1,1)$.  

	Clearly $\cB^u$ is bounded on $\ly{\fT}$.  Then the Marcinkiewicz interpolation theorem implies that $\cB^u$ is bounded on $\lp{\fT}$ for all 
	$p$ in $(1,\infty)$, as required.  
%

	Next we prove \rmii.  Consider a point~$o$ in $\fH_0$, and the function $\de_o$, which is equal to $1$ at $o$ and vanishes elsewhere.   
	For $x$ in $\fT$, denote by $\mod{x}$ the distance between $o$ and $x$.  If $x\in \fH_0$, then the smallest triangle that contains 
	both $x$ and $o$ is $T_{\mod{h(o\wedge x)}}(o\wedge x)$, where $o\wedge x$ denotes the \textit{confluent} of $o$ and $x$, viz. the point of least height
	that is a predecessor of both $o$ and $x$.  Note that 
	$2\, h(o\wedge x) = \mod{x}$.  Thus,
	$$
	\cB^u\!\de_o(x)
	= \frac{1}{\be\big(T_{\mod{h(o\wedge x)}}(o\wedge x))\big)}
	= b^{-\mod{x}/2}.
	$$
	Observe that the number of points in $\fH_0$ at distance $k$ from $o$ is equal to $1$ if $k=0$, and to $(b-1)\, b^{k/2-1}$ if $k$ is even.  
	Therefore 
	$$
	\int_{\fH_0}  \cB^u\!\de_o \wrt\mu
	=    \int_{\fH_0} b^{-\mod{x}/2} \wrt \mu(x)   
	= 1 + \frac{b-1}{b} \sum_{k\geq 2,\, \hbox{$k$\! \small{even}}}\, \, b^{-k/2} \, b^{k/2}  
	= \infty.  
	$$
	This proves \rmii, and concludes the proof of the theorem.
\end{proof}

\section{The uncentred triangular maximal operator}
\label{s: The uncentred}

Suppose that $\cG$ is a family of triangles in $\fT$.  A triangle $T$ in $\cG$ is \textit{maximal} in $\cG$ if $T' \in \cG$ and $T\neq T'$ imply that 
$T'\cap T \neq T$.  In other words, $T$ is maximal in $\cG$ with respect to the partial ordering induced by $\subseteq$.

Our proof of the $L^p$ boundedness of $\cU$ for $1<p<\infty$ is based on the following ``geometric'' lemma.  

\begin{lemma} \label{l: geometric lemma}
	Suppose that $\cG$ is a collection of maximal triangles in a locally finite tree $\fT$, with $\nu \geq 3$, and set 
	$\ds G := \bigcup_{T\in \cG} \, T$.  Then for every $r$ in $[1,\infty)$
	\begin{equation} \label{f: Lorentz}
	\Bignorm{\sum_{T\in\cG} \One_T}{r}
	\leq A_r \bignorm{\One_G}{r},
	\end{equation}
	where $\ds A_r := 4 \, \sum_{k=1}^\infty  k^r \, 2^{-k}$.
\end{lemma}

\begin{proof}
	Define the \textit{overlapping number} $\Om$ of the family $\cG$ by 
	$$
	\Om(x)
	:= \sharp \,\{T \in \cG: T\ni x\} 
	\quant x \in \fT.  
	$$
	For $x$ in $G$, denote by $T_1,\ldots,T_{\Om(x)}$ the (distinct) triangles in $\cG$ that contain~$x$, and by $v_1,\ldots,v_{\Om(x)}$
	their vertices.  By possibly relabelling the triangles, we can assume that the height of the vertices is a nonincreasing sequence, i.e.,
	$h (v_j) \geq h (v_{j+1})$, $j=1,\ldots,\Om(x)-1$.  In fact, this sequence is strictly decreasing.  Indeed, if 
	$h (v_j) = h (v_{j+1})$ for some $j$, then either $T_j\subseteq T_{j+1}$ or $T_{j+1}\subseteq T_j$, which would contradict
	the maximality of either $T_j$ or $T_{j+1}$.  Thus, $v_1\succ \ldots\succ v_{\Om(x)}$. 
	
	A similar argument shows that $b_1 > \ldots > b_{\Om(x)}$, where $b_j$ denotes the height (with respect to the point at infinity $\om$)
	of the points in $\be(T_j)$.  Therefore $T_1,\ldots, T_{\Om(x)}$ form a chain of triangles such that  
	$$
	h(v_1) > \ldots > h(v_{\Om(x)}) \geq h(x) \geq b_1 > \ldots > b_{\Om(x)}.  
	$$
	A moment's reflection then shows that $d\big(x,\be(T_{\Om(x)})\big) \geq \Om(x)-1$, and that $h(T_j) \geq \Om(x)-1$, $j=1.\ldots, \Om(x)$. 

	For every positive integer $k$ set $\Om_k := \big\{x \in G: \Om(x) = k\big\}$.  If $x\in \Om_k$, then~$x$ belongs to exactly $k$ triangles in $G$.  By
	the considerations above, the height of such triangles is $\geq k-1$, and there exists at least one of them, $T_x$ say, such that 
	$d\big(x,\be(T_x)\big) \geq k-1$.  In other words, $x$ belongs to 
	$$
	\bigcup_{m\geq k-1}^{h(T_x)} \,  p^m\big(\be(T_x)\big).
	$$
	Now, we let $x$ vary in $\Om_k$, and obtain 
	$$
	\Om_k \subseteq \bigcup_{T\in \cG: h(T) \geq k-1} \, \, \bigcup_{m\geq k-1}^{h(T)} \,  p^m\big(\be(T)\big).  
	$$
 	Notice that Lemma~\ref{l: triangles and bases} yields 
	$$
       		\Big|\bigcup_{m\geq k-1}^{h(T)} \,p^m\big(\be(T)\big)\Big| 
		\leq \sum_{m= k-1}^{h(T)}\, 2^{-m} \, \bigmod{\be(T)}
		\leq 2^{2-k} \, \bigmod{\be(T)}. 
	$$
	Hence 
	$$
	\bigmod{\Om_k} 
	\leq 2^{2-k} \, \sum_{T\in \cG} \, \bigmod{\be(T)}.  
	$$
	Since the triangles in $\cG$ are maximal, their bases are disjoint.  Therefore  
	$$
	\sum_{T\in \cG} \, \bigmod{\be(T)}
	= \Bigmod{\bigcup_{T\in \cG} \, \be(T)} 
	\leq \mod{G}.  
	$$
	Thus, 
   	\begin{equation}\label{claim}
		\bigmod{\Om_k}
		\leq  2^{2-k}\,  \mod{G}.
   	\end{equation} 
	Consequently, 
	$$
    		\int_{G} \Om(x)^r \wrt\mu(x) 
		=    \sum_{k=1}^\infty \, k^r\, \bigmod{\Om_k} 
		\leq 4 \, \sum_{k=1}^\infty  k^r \, 2^{-k} \, \mod{G} 
	$$
	which is equivalent to the required estimate.
\end{proof}

For notational convenience, for every $\al>0$ we shall denote the level set $E_{\cU f}(\al)$ also by $E(\al)$.   

\begin{remark} \label{rem: preliminaries cU}
\rm{
Observe that if $x\in E(\al)$, then there exists a triangle $T$ containing~$x$ such that  
\begin{equation} \label{f: average al}
\frac{1}{\bigmod{T}} \, \int_{T} \bigmod{f} \wrt \mu
> \al.  
\end{equation}
Then $T\subseteq E(\al)$.  This entails that $E(\al)$ can be written as a union of triangles~$T$ for which \eqref{f: average al} holds.  
Furthermore, if $T$ is one of these triangles and if $f \in \lp{\fT}$ for some $p$ in $(1,\infty)$, then \eqref{f: average al} and H\"older's inequality imply that 
\begin{equation} \label{f: p estimate}
\mod{T} < \frac{\bignormto{f}{p}{p}}{\al^p}.
\end{equation}
Now,  
$$
\mod{T} = \sum_{j=0}^{h(T)}\, \bigmod{s^j\big(v(T)\big)} \geq \sum_{j=0}^{h(T)} \, 2^j \geq 2^{h(T)};
$$
the first inequality above follows from the assumption $\nu\geq 3$.  Therefore
\begin{equation} \label{f: height in terms of al}
	h(T) \leq  \log_2 \mod{T} \leq \log_2 \frac{\normto{f}{p}{p}}{\al^p}.   
\end{equation}
Notice that $\diam (T) = 2 h(T)$ for every triangle $T$; thus, if it has nonempty intersection with $B_R(o)$, then
$T$ is contained in the ball $B_{R+2h(T)}(o)$.  If, in addition,~$T$ satisfies \eqref{f: p estimate}, then
then $T$ is contained in the ball with centre~$o$ and radius $R(\al) := R+2\log_2 \big(\normto{f}{p}{p}/\al^p\big)$.

In particular, if $f$ belongs to $\lp{\fT}$ for some $p<\infty$, then $E(\al)$ can be written as a union of a \textit{finite} number of triangles. 
}
\end{remark}

\begin{theorem}[Uncentred triangular maximal function] \label{t: uncentred triangular}
	Suppose that $\fT$ is a tree such that $\nu \geq 3$.  The following hold:
	\begin{enumerate}
		\item[\itemno1]
			the uncentred triangular maximal operator $\cU$ is bounded on $\lp{\fT}$ for every $p$ in $(1,\infty]$;
		\item[\itemno2]
			 if $b\geq 2$, then $\cU$ is not of weak type $(1,1)$ on the homogeneous tree~$\fT_b$.  
	\end{enumerate}	
\end{theorem}

\begin{proof}
	First we prove \rmi.
	We shall show that for every $p$ in $(1,\infty)$ 
	\begin{equation} \label{f: weak p}
		\bigmod{\big\{x \in \fT: \cU f(x) > \al \big\}}
		\leq A_{p'}^p\, \frac{\normto{f}{p}{p}}{\al^p} 
		\quant \al >0 \quant f \in \lp{\fT}, 
	\end{equation}
	where $\ds A_{p'} =  \sum_{k=1}^\infty  k^{p'} \, 2^{-k}$. 
	The required result then follows from the Marcinkiewicz interpolation theorem by interpolating \eqref{f: weak p} with the trivial $L^\infty$ bound.

	Preliminarily observe that if $\al \geq \norm{f}{p}$ and $\emptyset \neq T\subseteq E(\al)$, then \eqref{f: p estimate} implies $\mod{T} = 0$,
	which is absurd.  Therefore  $E(\al)$ is empty for all $\al \geq \norm{f}{p}$.  

	Thus, we assume henceforth that $\al < \norm{f}{p}$.  By Remark~\ref{rem: preliminaries cU}, $E(\al)$ can be written as a union of a 
	\textit{finite} number of triangles.
	Denote by $\cF(\al)$ the collection of all triangles~$T$ that are maximal in $E(\al)$, i.e. that are not properly contained in any larger 
	triangle in $E(\al)$; thus, $\ds E(\al) = \bigcup_{T\in \cF(\al)} T$.  
	We prove \eqref{f: weak p}.  Much as in the proof of \cite[Proposition 1]{CF}, observe that 
%
	$$
	\begin{aligned}
		\bigmod{E(\al)} 
		 \leq  \sum_{T\in \cF(\al)} \mod{T} 
		 \leq \frac{1}{\al}\,  \sum_{T\in \cF(\al)}\,  \int_{T} \mod{f} \wrt \mu 
		 \leq \frac{1}{\al}\,  \int_{\fT}  \mod{f} \,\, \sum_{T\in \cF(\al)} \One_T\wrt \mu. 
	\end{aligned}
	$$
	Now H\"older's inequality and \eqref{f: Lorentz} (with $p'$ in place of $r$) 
	and Lemma~\ref{l: geometric lemma} (with $\cF(\al)$ in place of $\cG$ and $E(\al)$ in place of $G$) yield
	$$
	\begin{aligned}
		\bigmod{E(\al)} 
		\leq \frac{\norm{f}{p}}{\al}  \,\, \Bignorm{\sum_{T\in \cF(\al)} \One_T}{p'} 
		\leq A_{p'}\, \frac{\norm{f}{p}}{\al} \bignorm{\One_{E(\al)}}{p'}.  
	\end{aligned}
	$$
	Finally, notice that $\ds \bignorm{\One_{E(\al)}}{p'} = \bigmod{E(\al)}^{1/p'}$, so that the last inequality may be rewritten as
	$$
	\bigmod{E(\al)} 
	\leq A_{p'}^p\, \al^{-p}\bignormto{f}{p}{p},
	$$
	as claimed.


	Next we prove \rmii.  
	Suppose that $T$ is a triangle in $\fT_b$ with vertex $x$ and height~$h$.  Note the following relation between $h$ and the volume of $T$: 
	\begin{equation} \label{f: volume and height}
	\bigmod{T} 
	=    \sum_{j=0}^h\, \bigmod{s^j(x)} 
	=    \frac{b^{h+1}-1}{b-1}.  
	\end{equation}

	Consider the unit point mass $\de_o$ at the point $o$. 
	We shall show that $\cU\de_o$ does not belong to weak $\lu{\fT_b}$.  Let $\al>0$.  
	Clearly $E_{\cU_{\de_o}}(\al)$ can be written as the union of maximal triangles.  Each such triangle $T$ satisfies 
	\begin{equation} \label{f: dal basso}
		\frac{1}{(b+1)\al}\le \mod{T} < \frac{1}{\al}.
	\end{equation}
	Indeed, the right hand inequality is a direct consequence of the fact that $T$ is contained in $E_{\cU_{\de_o}}(\al)$.  As to the left inequality,
	let $x$ and $h$ be the vertex and the height of $T$, respectively, and consider the triangle $\wt T$ with vertex $x$ and height $h+1$.  Since $T$ is 
	maximal, $\wt T$ is not contained in $E_{\cU_{\de_o}}(\al)$, whence $1/\bigmod{\wt T} \leq \al$.  Furthermore, \eqref{f: volume and height} implies 
	that $ \bigmod{\wt T} \leq (b+1) \, \bigmod{T}$.   The left hand inequality in \eqref{f: dal basso} follows by combining these two inequalities.

	Denote by $h_\al$ the largest integer such that a triangle $T$ in $\fT_b$ with height~$h_\al$ satisfies the right hand inequality in \eqref{f: dal basso}.
	If $T$ contains $o$, then $T$ is a maximal triangle in $E_{\cU_{\de_o}}(\al)$ and therefore it satisfies also the left hand inequality in 
	\eqref{f: dal basso}.   A simple calculation then shows that $b^{h_\al} \geq 1/(3b\al)$. 

	It is straightforward to see that the triangles with vertices $o,p(o),\ldots, p^{h_\al}(o)$ and of height~$h_\al$ are contained in $E_{\cU\de_o}(\al)$.  
	Thus,
	$$
	E_{\cU\de_o}(\al)
	\supset T_{h_\al}(o) \cup \, \bigcup_{k=1}^{h_\al}\big(T_{h_\al}(p^k(o))\setminus T_{h_\al}(p^{k-1}(o)) \big). 
	$$
	Notice that 
	$$
	\bigmod{T_{h_\al}(p^k(o))\setminus T_{h_\al}(p^{k-1}(o))}
	= 1 + (b-1) \, \sum_{j=0}^{h_\al-1} \, b^j
	= b^{h_\al}.  
	$$
	Therefore if $\al$ belongs to $\big(0,1/(3b)\big)$, then 
	\begin{equation} \label{f: lower bound} 
	\bigmod{E_{\cU\de_o}(\al)}
		\geq \frac{b^{h_\al+1}-1}{b-1} + \sum_{k=1}^{h_\al} \, b^{h_\al} 
		\geq {h_\al} \, b^{h_\al}
		\geq \frac{1}{3b\al} \, \log_b \frac{1}{3b\al}.   
	\end{equation}
	Thus, $\cU\de_o$ does not belong to weak $L^1$, as required.  
\end{proof}

\section{Further comments and exotic maximal operators}
\label{s: Further}

Theorem~\ref{t: uncentred triangular} raises the question of finding an endpoint result for~$\cU$ when~$p=1$.  With some effort, we can prove 
the following estimate on the homogeneous tree $\fT_b$, $b\geq 2$. 

\begin{theorem} \label{t: LlogL}
	There exists a constant $C$ such that 
	$$
	\bigmod{E_{\cU f}(\al)}
	\leq  C\, \frac{\norm{f}{1}}{\al} \, \log_b\Big(1+\frac{\norm{f}{1}}{\al}\Big)
	\quant \al>0 \quant f\in \lu{\fT_b}.  
	$$
\end{theorem}

\noindent
We believe that this estimate is not very interesting, for it seems not strong enough to imply the boundedness of $\cU$ on $\lp{\fT_b}$ for $p>1$.
For this reason, we omit the proof of Theorem~\ref{t: LlogL}. 

Observe that an estimate of the form 
	$$
	\bigmod{E_{\cU f}(\al)}
	\leq  C\, \frac{\norm{f}{1}}{\al} \, \log_b\Big(1+\frac{1}{\al}\Big)
	\quant \al>0 \quant f\in \lu{\fT_b},
	$$
which would imply the boundedness of $\cU$ on $\lp{\fT_b}$ for $p>1$, fails.

Indeed, let $o$ be a point in $\fT_b$, and consider $n\de_o$, where $n$ is a positive integer.  Observe that $E_{\cU(n\de_o)}(\al) = E_{\cU(\de_o)}(\al/n)$.
If the above estimate held, we would have
$$
\bigmod{E_{\cU \de_o}(\al/n)}
\leq  C\, \frac{n}{\al} \, \log_b\Big(1+\frac{1}{\al}\Big).
$$
By \eqref{f: lower bound}, the left hand side is bounded below by $\ds C\, \big(n/\al\big) \, \log_b \big(n/\al\big)$, at least for $\al$ small,
which is clearly incompatible with the upper bound above when $n$ tends to infinity.  

Recall that a fairly common strategy to prove weak type $(1,1)$ estimates for the ``global part'' of the HL maximal operator on manifolds 
with exponential volume growth is to majorize the maximal function with an appropriate integral operator, and prove that the latter is of weak type $(1,1)$.
See, for instance, \cite{Str}, where this strategy is shown to be effective in the study of the centred HL maximal function
on symmetric spaces of the noncompact type, and \cite{CMS1} for the case of homogeneous trees.

We shall prove that a similar approach fails for the uncentred triangular maximal operator $\cU$, even on the homogeneous tree $\fT_b$, $b\geq 2$.  
Consider the kernel
\begin{equation} \label{f: kappa}
\kappa(x,y) 
:= \sup_{T\ni x} \, \frac{\One_T(y)}{\mod{T}}
\quant x,y \in \fT_b,
\end{equation}
and denote by $\cK$ the corresponding integral operator, defined by 
$$
\cK f(x) 
:= \int_{\fT_b} \kappa(x,y) \, f(y) \wrt \mu(y)
\quant x \in \fT_b,
$$
where $f$ is any reasonable function on $\fT_b$.  Notice that $\cU f \leq \cK \mod{f}$.  The following result implies that $\cU$ and $\cK$ have
a quite different boundedness properties as operators acting on $\lp{\fT_b}$.

\begin{proposition}  \label{p: cK}
	The operator $\cK$ is unbounded on $\lp{\fT_b}$ for every $p$ in $[1,\infty]$ and for every $b\geq 2$.
\end{proposition}

\begin{proof}
	It is straightforward to check that the smallest triangle that contains two points $x$ and $y$ is the triangle with vertex $x\wedge y$
	(see the proof of Theorem~\ref{t: main cB}~\rmii\ for the notation) and height 
	$$
	\eta(x,y)
	:= \max \big(d(x,x\wedge y), d(y,x\wedge y)\big).
	$$   
	Clearly $\ds \eta(x,y) = \frac{1}{2} \, \big[d(x,y) + \mod{h(x)-h(y)}\big]$.  
	From the definition of~$\kappa$ (see \eqref{f: kappa}) and Lemma~\ref{l: triangles and bases}~\rmii\ we deduce that 
	$$
	\kappa(x,y)
	= \frac{1}{\bigmod{T_{x,y}}} \geq \frac{1}{2\,\bigmod{\be\big(T_{x,y}\big)}} \geq \frac{1}{2}\,\,  b^{-\eta(x,y)}
	\quant x,y\in \fT_b.
	$$
	Suppose that $o$ is a point in $\fH_0$, and consider, for each positive integer~$n$, the set $E_n := s^n(p^n(o))$, which is
	the base of the triangle with vertex $p^n(o)$ and height $n$.  Observe that for every $x$ and $y$ in $E_n$ we have $\eta(x,y) = d(x,y)/2$, so that
	$$
    	\cK \One_{E_n}(x)
	\geq \frac{1}{2}\, \int_{E_n} b^{-d(x,y)/2} \wrt \mu(y).
	$$
	Note that for every positive integer $j\leq n$ there are exactly $(b-1)\,b^{j-1}$ points in~$E_n$ at distance $2j$ from $x$. 
	Therefore the last integral can be rewritten as 
	$$
	1+\frac{b-1}{b}\, \sum_{j=1}^n b^{-j}\, b^j 
	= 1+\frac{b-1}{b}\,n.  
	$$
	Altogether  
	$$
		\cK \One_{E_n}(x) \geq \frac{b-1}{2b} \, n
		\quant x \in E_n,   
	$$
	from which the desired result for $p=\infty$ follows directly.  

	Now, set $C_b := (b-1)/(2b)$, and observe that if $p<\infty$, then for every positive integer $n$ the previous inequality yields
	$$
	\bignormto{\cK \One_{E_n}}{p}{p} 
	\geq C_b^p  \, n^p \, |E_n| 
	=    C_b^p  \, n^p \bignormto{\One_{E_n}}{p}{p}, 
	$$
	which implies that $\cK$ is unbounded on $\lp{\fT_b}$, as required.
\end{proof} 

It is worth observing that replacing triangles with appropriate slightly larger sets in the definition of $\cT$ and $\cU$ may yield significant modifications 
of the boundedness properties of the corresponding maximal operators, as we presently show. This is a further example that illustrates how sensitive are
maximal operators to the shape of the sets with respect to which we take averages.   

For every nonnegative integer $r$ consider the \textit{modified triangle} $T'_r(x) := T_r(x)\cup p^r(x)$, and the corresponding centred 
and uncentred maximal operators
$$
\cT' f(x)=\sup_{r \in \BN} \frac{1}{|T'_r(x)|}\int_{T'_r(x)} |f| \, \wrt \mu
\quad\hbox{and}\quad 
\cU'f(x)=\sup_{T' \ni x} \frac{1}{|T'|}\int_{T'} |f| \, \wrt \mu,
$$
where $T'$ is any modified triangle containing $x$.  We emphasize that $T'$ is obtained from a triangle $T$ by adjoining just a point at distance 
$h(T)$ from the vertex of $T$.  Observe that if there exists a positive constant~$C$ such that $\bigmod{T_r(x)}\geq C \, \bigmod{B_r(x)}$ 
for every triangle $T_r(x)$ in $\fT$, then 
\begin{equation} \label{f: modified max}
\cT' \leq C \, \cM
\qquad\hbox{and}\qquad
\cU' \leq C \, \cN.  
\end{equation}
For instance, this happens if $\fT = \fT_b$, or $\fT=\fS_{a,b}$ and $a\leq b<a^2$: here $\fS_{a,b}$ denotes 
the tree such that each vertex has either $a+1$ or $b+1$ neighbours according to the fact that its height is $< 1$ or $\geq  1$. 
We refer the reader to \cite{LMSV} for more on~$\fS_{a,b}$.
For each pair $a,b$ of positive integers, we denote the number $\log_a b$ by~$\tau$.
In the next proposition we show that there are trees where $\cT'$ and $\cU'$ have different boundedness properties than $\cT$ and $\cU$,
respectively.

\begin{proposition}\label{counterexamples} 
	The following hold:
	\begin{enumerate}
		\item[\itemno1]
			the operator $\cT'$ is bounded on $\lp{\fT_b}$ for every $p$ in $(1,\infty]$, and unbounded on $\lu{ \fT_b}$; 
		\item[\itemno2]
			if $a<b<a^2$, then $\cT'$ is bounded $L^p(\fS_{a,b})$ for $p>\tau$ and it is unbounded on $L^p(\fS_{a,b})$ for $p<\tau$;
		\item[\itemno3]
			the operator $\cU'$ is bounded on $\lp{\fT_b}$ if and only if $p > 2$.
	\end{enumerate}
\end{proposition} 

\begin{proof}
	The $L^p$ boundedness of $\cT'$ and $\cU'$ in the ranges described in \rmi-\rmiii\ above follow from the bounds \eqref{f: modified max} 
	and the positive results for $\cM$ and $\cN$ proved in \cite{NT,CMS1,LMSV,V}.

	Next we prove that $\cT'$ is unbounded on $\lu{\fT_b}$.  Fix a point $o$ in $\fH_0$, and consider the set $E := \{x \in \fT_b: o \succeq x\}$.  
	Clearly $E$ is the infinite triangle with vertex $o$.  It is straightforward to check that for each $x$ in $E$ 
	$$
    	\cT'\delta_o(x)
	= \frac{1}{\bigmod{T'_{\mod{x}}(x)}}. 
        $$
	By Lemma~\ref{l: triangles and bases}~\rmii, $\bigmod{T'_{\mod{x}}(x)} = \bigmod{T_{\mod{x}}(x)} +1 \leq 2\, b^{|x|} +1$, so that
	$$
	\bignorm{\cT'\de_o}{\lu{\fT_b}} 
	\geq  \int_E \cT'\de_o \wrt \mu 
	\geq  \sum_{j=0}^\infty \, \, \int_{E\cap \fH_{-j}} \frac{1}{2b^{j}+1} \wrt \mu.
    	$$
	Since $\bigmod{E\cap \fH_{-j}} = b^j$, the series above is not convergent, and the unboundedness of $\cT'$ on $\lu{\fT_b}$ follows, thereby 
	completing the proof of \rmi.

	To complete the proof of \rmii, fix a point $o$ in $\fH_0$, and for each positive integer $n$ consider the set 
	$E_n := s^n\big(p^n(o)\big)$, which is a subset of the horocycle~$\fH_0$ in $\fS_{a,b}$.  By Lemma~\ref{l: triangles and bases}~\rmii, and
	the fact that each vertex with nonpositive height has exactly $a$ successors, 
	$$
	\cT' \delta_{p^n(o)}(x)
	= \frac{1}{|T'_n(x)|}
	\geq \frac{1}{2 \, a^{n} +1 } 
	\quant x\in E_n,
    	$$
	whence, much as above, 
	$$
	\bignormto{\cT' \delta_{p^n(o)}}{\lp{\fS_{a,b}}}{p} 
	\geq \int_{E_n} \!\!\cT' \delta_{p^n(o)}^p \wrt \mu 
	\geq \frac{\bigmod{E_n}}{(2a^{n} +1)^p}.
	$$
	Observe that $\bigmod{E_n} = b^n = a^{\tau n}$.  Altogether, we see that 
	$$
	\bignormto{\cT' \delta_{p^n(o)}}{\lp{\fS_{a,b}}}{p} 
	\geq \frac{a^{\tau n}}{(2a^{n} +1)^p}.
	$$
	Since, by assumption, $p<\tau$, the right hand side above cannot be bounded with respect to $n$, and the desired result follows.  

	Finally we complete the proof of \rmiii\ by showing that $\cU'$ is unbounded on $\lp{\fT_b}$ for every $p\leq 2$.  Let 
	$E := \{x \in \fT_b: o \succeq x\}$.  If $x\in E$ and $d(o,x)$ is even, then 
	$$
	\cU' \delta_o(x) 
	= \frac{1}{\bigmod{T_{\mod{x}/2}'\big(p^{\mod{x}/2}(x)\big)}},  
	$$ 
	and Lemma~\ref{l: triangles and bases}~\rmii\ implies that $\bigmod{T_{\mod{x}/2}'\big(p^{\mod{x}/2}(x)\big)}\leq 2\, b^{|x|/2}+1$.  Thus, 
	$$
	\bignormto{\cU' \de_o}{\lp{\fT_b}}{p} 
	\geq \sum_{j=0}^\infty \,\int_{E_n\cap \fH_{-2j}}\!\! (2b^j+1)^{-p} \wrt \mu 
	=    \sum_{j=0}^\infty \,b^{2j}\, (2b^j+1)^{-p} .
	$$
	The required conclusion follows from the fact that for every $p\leq 2$ the series above is not convergent.
\end{proof}

\begin{remark}\label{rem: counterex infinite}
Finally, we present an example of a tree $\fT$ with unbounded geometry, where~$\cT'$ and, \textit{a fortiori} $\cU'$, is unbounded on $L^p$ for every $p<\infty$.

Let $\fT$ be the tree characterised by the property that each vertex off $\fH_0$ has three neighbours, and $\nu(x_j)=j+2$ where $\{x_j: j \ge 1\}$ 
is an enumeration of the points of $\fH_0$.   

Notice that for every $j \ge 1$
$$
	\cT'\de_{x_j}(y)
	= \frac{1}{|T_1'(y)|}
	= \frac{1}{4}
	\quant y\in s^1(x_j).  
$$
Therefore
$$
	\bignormto{\cT'\de_{x_j}}{p}{p} 
	\ge \sum_{y \in s^1(x_j)}  \!\! \cT'\de_{x_j}(y)^p
	=   \frac{1}{4^p} \, \mod{s^1(x_j)}
	=   \frac{j+2}{4^p}.
$$
Since $\|\de_{x_j}\|_p=1$, the operator norm of $\cT'$ on $\lp{\fT}$ is at least $(j+2)^{1/p}/4$.  By letting $j$ vary we obtain the required conclusion. 

Since $\cU' \ge \cT'$ pointwise, $\cU'$ is unbounded on $L^p(\fT)$ for every $p \in [1,\infty)$. 
\end{remark}

It is straightforward to check that for each $r>1$ there is no constant $C$ such that 
\begin{equation} \label{f: CF fails} 
	\Bignorm{\sum_{T'\in\cG'} \One_{T'}}{r}
	\leq C \bignorm{\One_{G'}}{r}
\end{equation}
for every finite family $\cG'$ of maximal modified triangles in $\fT$.  Here $G'$ is the union of the modified triangles in $\cG'$.

Indeed, it suffices to consider, for every positive integer $j$, the family $\cG_j'$ of the modified triangles $\{T_1'(y): y \in s^1(x_j)\}$.  
Then the $r^{\mathrm{th}}$ power of the right hand side of \eqref{f: CF fails} is equal to $C^r\, \big(3(j+2)+1\big)$, whereas the 
$r^{\mathrm{th}}$ power of left hand side is equal to $3(j+2)+(j+2)^r$.  Thus, \eqref{f: CF fails} fails for large values of~$j$.


\end{document}